\documentclass[a4paper,leqno,final]{amsart}

\usepackage{ifdraft}
\ifdraft{\usepackage[draft]{showkeys}}{\usepackage[final]{showkeys}}

\usepackage[usenames,dvipsnames]{xcolor} 
\usepackage{todonotes}

\usepackage[T1]{fontenc}
\usepackage[utf8]{inputenc}

\usepackage[final]{microtype}

\usepackage{mathtools}
\mathtoolsset{mathic}

\usepackage{amsmath,amsfonts,mathscinet,eucal,amssymb,amsthm}
\usepackage{bm,extarrows,upgreek,tensor,mathrsfs,textcomp,comment,bbm,braket}

\usepackage[shortlabels]{enumitem}

\usepackage[bbgreekl]{mathbbol}

\usepackage{tikz}
\usetikzlibrary{patterns,matrix,through,arrows,decorations.pathreplacing,decorations.markings,shadows,shapes.geometric,positioning,calc,backgrounds,fit}

\tikzstyle directed=[postaction={decorate,decoration={markings,
    mark=at position #1 with {\arrow{>}}}}]
\tikzstyle rdirected=[postaction={decorate,decoration={markings,
    mark=at position #1 with {\arrow{<}}}}]

\tikzset{anchorbase/.style={baseline={([yshift=-0.5ex]current bounding box.center)}},
arrowinthemiddle/.style={postaction=decorate,decoration={markings,mark=at position 0.5 with {\arrow{>}}}},
arrowinthemiddlerev/.style={postaction=decorate,decoration={markings,mark=at position 0.5 with {\arrow{<}}}},
cross line/.style={preaction={draw=white,line width=4pt,-}},
int/.style={very thick},
zero/.style={thin,dotted},
uno/.style={thin},
aboverotated/.style={above,rotate=60,anchor=west},
belowrotated/.style={below,rotate=60,anchor=east}}
\newcommand{\midarrow}{node[midway,sloped,allow upside down] {\hspace{0.05cm}\tikz[baseline=0] \draw[->] (0,0) -- +(.001,0);}}
\newcommand{\midarrowrev}{node[midway,sloped,allow upside down] {\hspace{0.05cm}\tikz[baseline=0] \draw[-<] (0,0) -- +(.001,0);}}
\newcommand{\nendarrow}{node[near end,sloped,allow upside down] {\hspace{0.05cm}\tikz[baseline=0] \draw[->] (0,0) -- +(.001,0);}}

\newcommand{\nendarrowrev}{node[near end,sloped,allow upside down] {\hspace{0.05cm}\tikz[baseline=0] \draw[-<] (0,0) -- +(.001,0);}}

\newcommand{\startarrowrev}{node[at start,sloped,allow upside down] {\hspace{0.05cm}\tikz[baseline=0] \draw[-<] (0,0) -- +(.001,0);}}

\newcommand{\midarrowint}{node[midway,sloped,allow upside down] {\hspace{0.05cm}\tikz[baseline=0] \draw[int,->] (0,0) -- +(.001,0);}}

\newcommand{\nendarrowint}{node[near end,sloped,allow upside down] {\hspace{0.05cm}\tikz[baseline=0] \draw[int,->] (0,0) -- +(.001,0);}}

\usepackage{graphicx}
\usepackage{wrapfig}
 \usepackage[font=sf, labelfont={sf,bf}, margin=1cm]{caption}
\usepackage{subcaption}

 \usepackage[strict]{changepage}

\usepackage{graphicx}

\usepackage[final=true]{hyperref,bookmark}

\numberwithin{equation}{section}

\usepackage{etoolbox}
\makeatletter
\let\ams@starttoc\@starttoc
\makeatother
\usepackage{parskip}
\makeatletter
\let\@starttoc\ams@starttoc
\patchcmd{\@starttoc}{\makeatletter}{\makeatletter\parskip\z@}{}{}
\makeatother
\newtheoremstyle{myplain} {6pt plus 6pt minus 2pt}
{6pt plus 6pt minus 2pt}
{\itshape}
{}
{\bfseries}
{.}
{.5em}
{}

\theoremstyle{myplain}
\newtheorem{theorem}{Theorem}[section]
\newtheorem*{theorem*}{Theorem}
\newtheorem{lemma}[theorem]{Lemma}
\newtheorem{prop}[theorem]{Proposition}
\newtheorem{corollary}[theorem]{Corollary}

\newtheoremstyle{mydefinition} {6pt plus 6pt minus 2pt}
{6pt plus 6pt minus 2pt}
{\itshape}
{}
{\bfseries}
{.}
{.5em}
{}

\theoremstyle{mydefinition}
\newtheorem{definition}[theorem]{Definition}

\newtheoremstyle{myexample} {6pt plus 6pt minus 2pt}
{6pt plus 6pt minus 2pt}
{}
{}
{\scshape}
{.}
{.5em}
{}

\theoremstyle{myexample}

\newtheoremstyle{myremark} {6pt plus 6pt minus 2pt}
{6pt plus 6pt minus 2pt}
{}
{}
{\scshape}
{.}
{.5em}
{}

\theoremstyle{myremark}
\newtheorem{remark}[theorem]{Remark}

\begingroup
    \makeatletter
    \@for\theoremstyle:=mydefinition,myremark,myplain,myexample\do{%
        \expandafter\g@addto@macro\csname th@\theoremstyle\endcsname{%
            \addtolength\thm@preskip\parskip
            }%
        }
\endgroup

\DeclareSymbolFontAlphabet{\mathbb}{AMSb}
\DeclareSymbolFontAlphabet{\mathbbol}{bbold}
\DeclareMathAlphabet{\mathpzc}{OT1}{pzc}{m}{it}

\DeclareSymbolFont{usualmathcal}{OMS}{cmsy}{m}{n}
\DeclareSymbolFontAlphabet{\mathucal}{usualmathcal}

\newcommand{\Z}{\mathbb{Z}}

\newcommand{\C}{\mathbb{C}}

\newcommand{\gl}{\mathfrak{gl}}

\newcommand{\suchthat}{\mid} 
\newcommand{\mapto}{\rightarrow}

\newcommand{\into}{\hookrightarrow}

\DeclareMathOperator{\End}{End}

\DeclareMathOperator{\tr}{tr}

\newcommand{\abs}[1]{\left|#1\right|}

\renewcommand{\epsilon}{\varepsilon}

\renewcommand{\phi}{\varphi}

\newcommand{\calG}{{\mathcal{G}}}

\newcommand{\calQ}{{\mathcal{Q}}}

\newcommand{\boldeta}{{\boldsymbol{\eta}}}

\newcommand{\Spider}{\mathsf{Sp}}

\newcommand{\fieldk}{{\mathbbm{k}}}

\newcommand{\ucalH}{{\mathucal H}}

\newcommand{\RT}{{\mathcal{RT}}}

\newcommand{\onepartition}{\tikz[baseline]{\draw (0,0) rectangle (1ex,1ex);}}

\newcommand{\var}{t}

\newcommand{\unknot}{\mathord{\bigcirc}}

\newcommand{\uBr}{\mathsf{Br}}

\newcommand{\catTangles}{{\mathsf{Tangles}}}
\newcommand{\lab}{{\ell\!\textit{a\hspace{-1pt}b}}}
\newcommand{\catRep}{{\mathsf{Rep}}}

\newcommand{\emptysequence}{\varnothing}

\usepackage{wrapfig}

\hypersetup{
  colorlinks = true,
  urlcolor = blue,
  pdfauthor = {Hoel Queffelec and Antonio Sartori},
  pdfkeywords = {Alexander polynomial, HOMFLY-PT polynomial, Jones polynomial, quantum link invariants, Schur-Weyl duality, Mirror symmetry},
  pdftitle = {A note on the gl(m|n) link invariants and the HOMFLY-PT polynomial},
  pdfsubject = {},
  pdfpagemode = UseNone,
  bookmarksopen = true,
  bookmarksopenlevel = 3,
  pdfdisplaydoctitle = true }

\newcommand{\empha}[1]{\emph{\textbf{\boldmath #1}}}

\title{A note on the \(\gl_{m|n}\) link invariants and the HOMFLY-PT polynomial}

\author{Hoel Queffelec}
\address{MSI, Australian National University, John Dedman Building, 27 Union Lane, Canberra ACT 0200}
\email{hoel.queffelec@anu.edu.au}
\thanks{H.Q.\ was funded by the ARC DP 140103821.}

\author{Antonio Sartori}
\address{Mathematisches Institut, Albert-Ludwigs-Universität Freiburg,
Eckerstraße 1, 79104 Freiburg im Breisgau, Germany}
\email{antonio.sartori@math.uni-freiburg.de}

\tikzset{smallnodes/.style={every node/.style={font=\footnotesize}}}

\begin{document}

\begin{abstract}
We present a short and unified representation-theoretical treatment of type A link invariants (that is, the HOMFLY-PT polynomials, the Jones polynomial, the Alexander polynomial and, more generally, the  \(\gl_{m|n}\) quantum invariants) as link invariants with values in the quantized oriented Brauer category.
\end{abstract}
\maketitle

\section{Introduction}
\label{sec:introduction}

The HOMFLY-PT polynomial \cite{HOMFLY,PT} is a 2-variable polynomial link invariant generalizing the Jones polynomial \cite{Jones}, the Alexander polynomial \cite{MR1501429} and the \(\mathfrak{sl}_k\) Reshetikhin-Turaev link invariant \cite{MR1036112}. 
One can also label the strands of the links by partitions, and obtain colored versions of these polynomials. All these link invariants have the common property of arising from the representation theory of the Lie superalgebra \(\gl_{m|n}\), and because of this we call them \emph{invariants of type A}.

In more detail, link invariants of type A can be constructed as \(U_q(\gl_{m|n})\)--equivariant homomorphisms. In this short note, we present a unified approach to such link invariants  by seeing them as invariants with  values in the quantized oriented Brauer category,  a universal category describing intertwiners of \(U_q(\gl_{m|n})\)--representations. We also define reduced link invariants by  the usual trick of cutting open one of the strands.

This approach enables to give a very easy construction of  link invariants of type A, which does not directly require the knowledge of quantum group R-matrices. Moreover, this  yields short, clean and self-contained proofs of some well-known folklore invariance and symmetry properties, which are particularly interesting also for categorification problems (see \cite{2013arXiv1304.3481G}). In particular, we prove that colored \(\gl_{m|n}\) polynomials of links only depend on the difference \(d=m-n\) (see 
 Theorem~\ref{thm:1} and, in the reduced case, Proposition~\ref{prop:2}), and we use an automorphism of the Brauer category to prove the so-called mirror symmetry of the HOMFLY-PT polynomial (Theorem~\ref{thm:2}).

\section{Notation and conventions}
\label{sec:overview}

We will work over a ring \(\fieldk\) containing the field \(\C(q)\) and an element \(q^\var\) which is not a root of  unity. In particular, we will consider two possibilities:
\begin{enumerate}[label=(\alph*),topsep=-1ex,itemsep=-1ex,partopsep=0ex,parsep=1ex]
\item \label{item:1} \(\var=\beta\) and \(\fieldk=\C(q)[q^{\pm \beta}]\),
  i.e.\ an extension of \(\C(q)\)
  by a formal variable \(q^\beta\), or 
\item \label{item:2} \(\var=d \in \Z\) and \(\fieldk=\C(q)\).
\end{enumerate}
In the first case, we will say that \(\var\), or \(\beta\),
  is \emph{generic}. As a convention, we will use the letter \(\var\) for encompassing both cases above, while we will use \(\beta\) when we will assume that we are in case~\ref{item:1} and we will use \(d\) when we will assume that we are in case~\ref{item:2}.
For \(x \in \Z \var + \Z\) we define
\begin{equation}
  \label{eq:15}
  [x]  = \frac{q^x - q^{-x}}{q - q^{-1}}.
\end{equation}

\subsection{Tangles}
\label{sec:tangles}

 Let \(\catTangles\) be the monoidal category 
of oriented framed tangles, whose objects are sequences of orientations \(\{\uparrow,\downarrow\}\) and whose morphisms are oriented framed tangles modulo isotopy.
Let also \(\catTangles_\fieldk\) be its \(\fieldk\)--linear version, with the same objects but with morphisms being  \(\fieldk\)--vector spaces
\begin{equation}
  \label{eq:2}
  \catTangles_{\fieldk}(\boldeta,\boldeta') = \operatorname{span}_{\fieldk} \catTangles(\boldeta,\boldeta').
\end{equation}

\begin{wrapfigure}{r}{2cm}%
\vspace{-0.45cm}
 \begin{tikzpicture}[scale=0.4,every
    node/.style={font=\footnotesize},anchorbase]
    \draw (0,0) rectangle (5,-1); \draw (0,-1) rectangle
    (4,-2); \draw (0,-3.5) rectangle (2,-4.5); \node at
    (1,-2.5) {$\vdots$}; \node at (2.5,-0.5)
    {$\uplambda_1$}; \node at (2,-1.5) {$\uplambda_2$};
    \node at (1,-4) {$\uplambda_\ell$};
  \end{tikzpicture}%
\vspace{-0.3cm}
\end{wrapfigure}
 \subsection{Partitions}
\label{sec:partitions}

  We denote by \(\lambda \dashv N\) a partition \(\lambda\) of \(N \geq 0\), which is  a non-increasing sequence  \(\lambda=(\lambda_1,\lambda_2,\dotsc)\)  of non-negative integers such that  \(\abs{\lambda} = \sum \lambda_i = N\).
  The transposed partition \(\lambda^\top\)  is defined by  \(\lambda^\top_i = \# \{h \suchthat \lambda_h \geq  i\}\).
  Partitions are usually identified with Young diagrams, as in the picture on the right. The  only partition of \(0\)  is the empty partition \(\varnothing\),  and the only partition of \(1\) is the one-box partition \(\onepartition\).

\subsection{Labeled tangles}
\label{sec:labeled-tangles}

 We let \(\catTangles^\lab\) be the category of oriented framed tangles whose connected components are labeled by partitions. We will regard a morphism in \(\catTangles^\lab\) as a pair consisting of a tangle \(T\) and a labeling \(\ell\) of its strands. Given a tangle \(T\) and a partition \(\lambda\), we denote by \((T,\underline \lambda)\) the labeled tangle \(T\) such that all strands are labeled by \(\lambda\). There is an obvious inclusion \(\catTangles \into \catTangles^\lab\), given by \(T \mapsto (T, \underline \onepartition)\). We will sometimes use \emph{color} as a synonymous of \emph{label}.

\section{The quantized oriented Brauer category}
\label{sec:quant-wall-brau}

We recall the definition of the quantized oriented Brauer algebra/category, following \cite{MR3181742}.

\begin{definition}\label{def:2}
  The \empha{quantized oriented Brauer category} \(\uBr(t)\) is the quotient of \(\catTangles_\fieldk\)
   modulo the following relations
\begin{subequations}
  \label{eq:5}
  \begin{align}
     \begin{tikzpicture}[smallnodes,anchorbase,xscale=0.7,yscale=0.5]
   \draw[uno] (1,0)  -- ++(0,0.3) \midarrow .. controls ++(0,0.7) and ++(0,-0.7) .. ++(-1,1.4) -- ++(0,0.3)  \midarrow ;
   \draw[uno,cross line] (0,0)  -- ++(0,0.3) \midarrow .. controls ++(0,0.7) and ++(0,-0.7) .. ++(1,1.4) -- ++(0,0.3)  \midarrow ;
    \end{tikzpicture} \; - \;
    \begin{tikzpicture}[smallnodes,anchorbase,xscale=0.7,yscale=0.5]
   \draw[uno] (0,0)  -- ++(0,0.3) \midarrow .. controls ++(0,0.7) and ++(0,-0.7) .. ++(1,1.4) -- ++(0,0.3)  \midarrow ;
   \draw[uno,cross line] (1,0)  -- ++(0,0.3) \midarrow .. controls ++(0,0.7) and ++(0,-0.7) .. ++(-1,1.4) -- ++(0,0.3)  \midarrow ;
    \end{tikzpicture} \; & = (q^{-1}-q) \;
    \begin{tikzpicture}[smallnodes,anchorbase,yscale=0.5,xscale=0.7]
   \draw[uno] (1,0) .. controls ++(0.25,0.5) and ++(0.25,-0.5)  .. ++(0,2) \midarrow ;
   \draw[uno] (2,0) .. controls ++(-0.25,0.5) and ++(-0.25,-0.5)  .. ++(0,2) \midarrow  ;
    \end{tikzpicture}\;,  \qquad&
        \begin{tikzpicture}[smallnodes,anchorbase,scale=0.7]
      \draw[uno] (0,0) arc (0:360:0.5cm) \midarrowrev;
    \end{tikzpicture} \; =\;
        \begin{tikzpicture}[smallnodes,anchorbase,scale=0.7]
      \draw[uno] (0,0) arc (0:360:0.5cm) \midarrow;
    \end{tikzpicture} \; &= [\var], \label{eq:52}\\
    \begin{tikzpicture}[smallnodes,anchorbase,xscale=0.7,yscale=0.5]
   \draw[uno] (0.8,1)  .. controls ++(0,-0.3) and ++(0.3,0) .. ++(-0.3,-0.6)  .. controls ++(-0.5,0) and ++(0,-0.7) .. ++(-0.5,1.6) \nendarrow  ;
   \draw[uno, cross line] (0.8,1)  .. controls ++(0,0.3) and ++(0.3,0) .. ++(-0.3,0.6) \startarrowrev .. controls ++(-0.5,0) and ++(0,0.7) .. ++(-0.5,-1.6) \nendarrowrev;
    \end{tikzpicture} \; = \;
    \begin{tikzpicture}[smallnodes,anchorbase,xscale=0.7,yscale=0.5]
   \draw[uno, cross line] (-0.8,1)  .. controls ++(0,0.3) and ++(-0.3,0) .. ++(0.3,0.6) \startarrowrev .. controls ++(0.5,0) and ++(0,0.7) .. ++(0.5,-1.6) \nendarrowrev ;
   \draw[uno, cross line] (-0.8,1)  .. controls ++(0,-0.3) and ++(-0.3,0) .. ++(0.3,-0.6)  .. controls ++(0.5,0) and ++(0,-0.7) .. ++(0.5,1.6) \nendarrow  ;
    \end{tikzpicture} \;& = q^{-\var} \;
    \begin{tikzpicture}[smallnodes,anchorbase,xscale=0.7,yscale=0.5]
      \draw[uno] (0,0)  -- ++(0,2) \midarrow ;
    \end{tikzpicture}\;, &
    \begin{tikzpicture}[smallnodes,anchorbase,xscale=0.7,yscale=0.5]
   \draw[uno, cross line] (0.8,1)  .. controls ++(0,0.3) and ++(0.3,0) .. ++(-0.3,0.6) \startarrowrev .. controls ++(-0.5,0) and ++(0,0.7) .. ++(-0.5,-1.6) \nendarrowrev ;
   \draw[uno, cross line] (0.8,1)  .. controls ++(0,-0.3) and ++(0.3,0) .. ++(-0.3,-0.6)  .. controls ++(-0.5,0) and ++(0,-0.7) .. ++(-0.5,1.6) \nendarrow  ;
    \end{tikzpicture} \; = \;
    \begin{tikzpicture}[smallnodes,anchorbase,xscale=0.7,yscale=0.5]
   \draw[uno] (-0.8,1)  .. controls ++(0,-0.3) and ++(-0.3,0) .. ++(0.3,-0.6)  .. controls ++(0.5,0) and ++(0,-0.7) .. ++(0.5,1.6) \nendarrow  ;
   \draw[uno, cross line] (-0.8,1)  .. controls ++(0,0.3) and ++(-0.3,0) .. ++(0.3,0.6) \startarrowrev .. controls ++(0.5,0) and ++(0,0.7) .. ++(0.5,-1.6) \nendarrowrev;
    \end{tikzpicture} \; & = q^{+ \var} \;
    \begin{tikzpicture}[smallnodes,anchorbase,xscale=0.7,yscale=0.5]
      \draw[uno] (0,0)  -- ++(0,2) \midarrow ;
    \end{tikzpicture}.\label{eq:54}
  \end{align}
\end{subequations}
\end{definition}

Being a quotient of the ribbon category \(\catTangles_\fieldk\), the category \(\uBr(\var)\) inherits a ribbon structure.
We denote by \(\calQ_t \colon \catTangles \mapto \uBr(\var)\) the composition of the inclusion \(\catTangles \into \catTangles_\fieldk\) with the quotient functor \(\catTangles_\fieldk \mapto \uBr(\var)\).
It is shown in \cite[lemma~2.4]{MR3181742}
that \(\End_{\uBr(t)}(\uparrow^{\otimes r}  \downarrow^{\otimes s})\) is free of rank \((r+s)!\). In particular, for the empty sequence \(\emptysequence\) we have \(\End_{\uBr(t)}(\emptysequence) \cong \fieldk \).

\subsection{The Hecke algebra}
\label{sec:hecke-algebra}

The endomorphism space \(\End_{\uBr(t)}(\uparrow^{\otimes r})\) is  the \empha{Hecke algebra} \(\ucalH_N\) over \(\fieldk\). 
As well known,  \(\ucalH_r\) is a finite-dimensional  semisimple algebra. Its finite-dimensional simple representations up to isomorphism are parametrized by partitions \(\lambda\) of \(r\), and we denote them by \(S(\lambda)\) for \(\lambda \dashv r\). In particular, \(\ucalH_r\) decomposes as
\begin{equation}
  \label{eq:3}
  \ucalH_r = \bigoplus_{\lambda \dashv\, r} e_\lambda \ucalH_r e_\lambda,
\end{equation}
where the \(e_\lambda\)'s are central idempotents, \(\ucalH_r e_\lambda \cong S(\lambda)^{\oplus \dim_\fieldk S(\lambda)}\) as a left module and \(e_\lambda \ucalH_r e_\lambda\) is isomorphic to a matrix algebra. We remark that one can write explicit formulas for the idempotents \(e_\lambda\), similarly as for the symmetric group, see \cite{MR873212}.

For each \(\lambda \dashv r\), we choose in \(e_\lambda \ucalH_r e_\lambda\) a primitive idempotent \(p_\lambda\). Unless \(e_\lambda\) itself is primitive (and this happens if and only if \(\lambda\) is a row or a column partition), the element \(p_\lambda\) is not uniquely determined. But any two choices are conjugated in \(e_\lambda \ucalH_r e_\lambda\), and hence also in \(\ucalH_r\).

\subsection{Cabling}
\label{sec:cabling-1}

We adopt the following graphical convention for picturing morphisms in the Brauer category: when we draw a thick strand labeled by \(\lambda\), this stands for \(\abs{\lambda}\) parallel strands, close to each other, with the idempotent \(p_\lambda\) somewhere on these strands:
\begin{equation}
  \label{eq:11}
  \begin{tikzpicture}[anchorbase]
    \draw[int] (0,0) -- ++(0,0.9) -- ++(0,0.1) \nendarrowint node[below right] {$\lambda$};
  \end{tikzpicture}  = \;
  \begin{tikzpicture}[anchorbase]
    \draw[uno] (0.1,0) -- ++(0,0.9) -- ++(0,0.1) \nendarrow ;
    \draw[uno] (0.9,0) -- ++(0,0.9) -- ++(0,0.1) \nendarrow ;
    \node[draw,fill=white,minimum width=1cm,minimum height=0.5cm] at (0.5,0.5) {$p_\lambda$};
    \node at (0.5,0.1) {$\cdots$};
    \node at (0.5,0.85) {$\cdots$};
  \end{tikzpicture}
\end{equation}
Since the Brauer category is ribbon, it does not matter at which point on the parallel strands we put the idempotent \(p_\lambda\), as we can slide it around. Of course, if convenient, we can also put more than one copy of \(p_\lambda\) along the strands.
This procedure allows us to define a monoidal functor \(\calQ_\var^\lab \colon \catTangles^\lab \mapto \uBr(\var) \).

\begin{remark}\label{rem:4}
  One can also make sense of this cabling procedure more formally in the category-theoretical setting. Namely, one can construct the Karoubi envelope  of the additive closure of \(\uBr(\var)\), as explained for example in \cite[section~2]{MR2998810}. This enlarged category is monoidally generated by primitive idempotents in \(\uBr(\var)\). In fact, for our purposes it would be sufficient to consider a partial Karoubi envelope \(\widetilde{\uBr}(\var)\) which is the smallest additive monoidal category containing all images of the primitive idempotents \(p_\lambda\). The ribbon structure of \(\uBr(\var)\) induces a ribbon structure on \(\widetilde{\uBr}(\var)\), and one obtains immediately a functor \(\calQ_\var^\lab \colon \catTangles^\lab \mapto \widetilde{\uBr}(\var)\).
\end{remark}

\section{Link invariants of type A}
\label{sec:link-invariants-type}

\begin{definition}\label{def:1}
  Let \(L\) be an oriented framed link.
  \begin{itemize}
  \item The \empha{HOMFLY-PT polynomial} of \(L\), denoted by \(P_\beta(L)\), is the image element \(\calQ_\beta(L) \in \End_{\uBr(\beta)}(\emptysequence) = \fieldk\).
  \item 
Let \(d \in \Z\). The 
\empha{\(d\)--polynomial} of \(L\), denoted by  \(P_d(L)\), is the image element \(\calQ_d(L) \in \End_{\uBr(d)}(\emptysequence) = \fieldk\).
  \end{itemize}
\end{definition}

\begin{remark}\label{rem:1}
  It follows  from the fact that \(\uBr(d)\) is defined over \(\C[q,q^{-1}]\) that \(P_{d}(L)\) is actually a Laurent polynomial in \(\C[q,q^{-1}]\). Similarly, one sees that \(P_\beta(L)\) is an element of \(\C\big[q,q^{-1},q^{\beta},q^{-\beta},[\beta]\big]\).
\end{remark}

\begin{definition}\label{def:5}
  Let \(L\) be an oriented framed link and \(\ell\) be a labeling of its strands.
  \begin{itemize}
  \item The \empha{\(\ell\)--labeled HOMFLY-PT polynomial} of \(L\), denoted by \(P^\ell_\beta(L)\), is the image element
\(\calQ^\lab_\beta( L,\ell ) \in \End_{\uBr(\beta)}(\emptysequence)= \fieldk\).
  \item 
Let \(d \in \Z\).
The \empha{\(\ell\)--labeled 
\(d\)--polynomial} of \(L\), denoted by \(P^\ell_d(L)\), is the image element
\(\calQ^\lab_d( L, \ell) \in \End_{\uBr(d)}(\emptysequence)= \fieldk\).
  \end{itemize}
\end{definition}

\begin{lemma}\label{lem:5}
  The definition above does not depend on the choice of the elements \(p_\lambda\).
\end{lemma}

\begin{proof}
Suppose \(\lambda \dashv N\), and let \(p_\lambda'\) another choice for a primitive idempotent in \(e_\lambda \ucalH_N e_\lambda\). Then there is an invertible element  \(x \in \ucalH_N\) such that \(x p_\lambda' x^{-1} = p_{\lambda}\). Since we can slide \(x\) around on the cabled strands and cancel it with \(x^{-1}\), the independence on the particular choice for \(p_\lambda\) follows.
\end{proof}

\begin{remark}\label{rem:2}
  In the general case, it is not clear to us how one can deduce 
from this definition that \(P_{m|n}^\ell(L)\) is a Laurent polynomial.
If \(\ell\) labels every strands by a one-column partition, then this can be deduced using the category \(\Spider(\beta)\), introduced by the two authors in \cite{QS2}, which is defined over \(\C[q,q^{-1}]\). One can argue analogously if \(\ell\) labels every strands by a one-row partition.
\end{remark}

Notice that it follows immediately that the \(\beta=d\) specialization of the (labeled) HOMFLY-PT polynomial yields the (labeled) \(d\)--polynomial. We stress that Definition~\ref{def:5} is just a reformulation of Reshetikhin-Turaev's construction:

\begin{prop}\label{prop:1}
  Let \(m,n \in \Z_{\geq 0}\)
 and let \(d=m-n\). Then \(P^\ell_d(L)\) is the labeled \(\gl_{m|n}\) link invariant (given by the Reshetikhin-Turaev construction).
\end{prop}

\begin{proof}
  Let  \(\catRep_{m|n}\) denote the category of finite-dimensional representations of the quantum group \(U_q(\gl_{m|n})\).
Notice that one can make sense of this also for \(m=n=0\): in this case, \(\gl_{0|0}\) is the trivial (zero-dimensional) Lie algebra, and \(\catRep_{0|0}\) is equivalent to the category of finite-dimensional \(\C(q)\)--vector spaces.  The Reshetikhin-Turaev functor \(\RT_{m|n} \colon \catTangles^\lab \mapto \catRep_{m|n}\) factors as
 \begin{equation}\label{eq:4}
   \begin{tikzpicture}[baseline=(current bounding box.center)]
     \matrix (m) [matrix of math nodes, row sep=2.5em, column
       sep=5em, text height=1.5ex, text depth=0.25ex] {
       & \uBr(d) \\
      \catTangles^\lab & \catRep_{m|n} \\};
    \path[->] (m-2-1.north) edge node[midway,above] {$\calQ^\lab_d$} (m-1-2);
    \path[->] (m-2-1) edge node[midway,below] {$\RT_{m|n}$} (m-2-2);
    \path[->] (m-1-2) edge node[midway,right] {$\calG_{m|n}$} (m-2-2);
   \end{tikzpicture}
 \end{equation}
 (For this, one only has to check that the relations \eqref{eq:5} are satisfied in \(\catRep_{m|n}\), and this is well-known, see for example \cite[section~3]{QS2} and references therein.) Moreover, 
it is easy to see that \(\calG_{m|n}\) induces an isomorphism between \(\End_{\uBr(d)}(\emptysequence)\) and \(\End_{U_q(\gl_{m|n})}(\C(q))\), which are  both  naturally identified with \(\C(q)\). Hence we have \(\calQ^\lab_d(L,\ell) = \RT_{m|n}(L,\ell)\).
\end{proof}

As an immediate consequence, we obtain the following well-known important result: 
\begin{theorem}\label{thm:1}
  The \(\gl_{m|n}\) Reshetikhin-Turaev invariant of links colored by partitions only depends on the difference \(m-n\).
\end{theorem}

In particular, \(P_{2}(L)\) is the \empha{Jones polynomial} of \(L\).

\subsection{Symmetry for the HOMFLY-PT polynomial}
\label{sec:symmetry-homfly-pt}

We conclude this section giving an easy proof of a well-known symmetry of the HOMFLY-PT polynomial, which follows immediately from the existence of an automorphism of the quantized oriented Brauer category.

\begin{theorem}\label{thm:2}
  Let \((L,\ell)\) be a labeled link, and let \(\ell^\top\) denote the transpose labeling (which labels each strands by the transpose partition). We have
  \begin{equation}
P^\ell_\beta(L)(q,q^\beta) = P^{\ell^\top}_\beta(L)(-q^{-1},q^\beta).\label{eq:13}
\end{equation}
\end{theorem}

\begin{proof}
  In the case \(\beta\) generic, we can define a \(\C\)--linear  involution \(\tau\) of the quantized walled Brauer category which fixes all tangle diagrams and which sends \(q \mapsto -q^{-1}\). It is immediate to check that the defining relations are satisfied (notice that, since \(\beta\) is generic, \(\tau\) fixes \(q^{ \beta}\)). It is well-known that by applying \(\tau\) to the Hecke algebra \(\ucalH_N\) one  interchanges the simple representations \(S(\uplambda)\) and \(S(\uplambda^T)\). In particular, \(\tau(p_\lambda)\) is conjugated to \(p_{\lambda^\top}\). 
  Hence \(P^{\ell^\top}_\beta(L)(q,q^\beta)=\tau(P^{\ell}_\beta(L)(q,q^\beta))=P^{\ell}_\beta(L)(-q^{-1},q^\beta)\).
\end{proof}

Another 
 proof of this symmetry (although in a slightly different formulation) has been given using web categories in \cite[proposition~4.4]{TVW} (see also references therein for older discussions).

For the usual (uncolored) HOMFLY-PT polynomial we get the following property:

\begin{corollary}
  We have \(P_\beta(L)(q,q^\beta) = P_\beta(L)(-q^{-1},q^\beta).\)
\end{corollary}

\begin{remark}\label{rem:5}
  From the proof above it is also absolutely clear why this symmetry holds for the HOMFLY-PT polynomial but fails for the \(d\)--polynomial.
\end{remark}

\section{Reduced link invariants of type A}
\label{sec:reduc-link-invar}

In this section we introduce reduced link invariants, following ideas from \cite{GPT} and related works. The main goal is to get non-trivial invariants also in the case \(d=0\), and in particular to define the Alexander polynomial. Indeed, we have:

\begin{lemma}\label{lem:1}
  If \(d=0\)
  then \(P^\ell_0(L) = 0\)
  for all links \(L\) and non-trivial labelings \(\ell\).
\end{lemma}

\begin{proof}
  This follows immediately from Theorem~\ref{thm:1}, since the \(\gl_{0|0}\) link invariant is zero if at least one strand is labeled by a non-empty partition.
\end{proof}

Given a link \(L\) with a labeling \(\ell\) of its strands and a chosen strand labeled by \(\lambda\), we can cut open this strand and obtain a tangle  \(T \in \End_{\catTangles^\lab}(1)\). The link \(L\) is obtained as closure of the tangle \(T\):
 \begin{equation}
   \label{eq:6}
   \begin{tikzpicture}[anchorbase]
     \node[draw,minimum width=1cm, minimum height=1cm] at (0,0) {$L$};
   \end{tikzpicture} \; =\;
   \begin{tikzpicture}[anchorbase,yscale=0.5]
     \draw[int] (0,-0.5) -- ++ (0,1) .. controls ++(0,1) and ++(0,1) .. ++(1,0) \midarrowint -- node[above right] {\footnotesize $\lambda$} ++(0,-1) .. controls ++(0,-1) and ++(0,-1) ..  (0,-.5);
     \node[fill=white,draw,minimum width=0.5cm, minimum height=0.5cm] at (0,0) {$T$};
   \end{tikzpicture}
\end{equation}
Notice that, since the Hecke algebra is semisimple and \(p_\lambda\) is a primitive idempotent, we have \(p_\lambda \End_{\uBr(\var)} (\uparrow^{\abs{\lambda}}) p_\lambda \cong \fieldk\). This allows us to give the following definition:

\begin{definition}\label{def:6}
  Let \(L\) be an oriented framed link and \(\ell\) be a labeling of its strands. Regard \(L\) as closure of a tangle \(T\) obtained by cutting open a strand labeled by \(\lambda\), as in \eqref{eq:6}.
  \begin{itemize}
  \item The \empha{\(\ell\)--labeled \(\lambda\)--reduced HOMFLY-PT polynomial} of \(L\), denoted by \(P^{\ell,\lambda}_\beta(L)\), is the scalar \(\alpha\) such that  \(\calQ^\lab_\beta(T,\ell) = \alpha p_\lambda \in p_\lambda \End_{\uBr(\beta)}(1^{\abs{\lambda}}) p_\lambda \).
  \item Let \(d \in \Z\).
The \empha{\(\ell\)--labeled \(\lambda\)--reduced \(d\)--polynomial} of \(L\), denoted by \(P^{\ell,\lambda}_{d}(L)\), is the scalar \(\alpha\) such that  \(\calQ_{d}^\lab(T,\ell) = \alpha p_\lambda \in p_\lambda \End_{\uBr(d)}(1^{\abs{\lambda}}) p_\lambda\).
  \end{itemize}
\end{definition}

Of course, we have to check that the definition does not depend on the chosen strand labeled by \(\lambda\) (and on the particular point we cut on the strand). This is implied by Lemma~\ref{lem:2}
below.
First,  notice  that by applying \(\calQ_\var^\lab\) to both sides of \eqref{eq:6} we get
\begin{equation}\label{eq:12}
P_\var^{\ell}(L) = \tr_{\uBr(\var)} p_\lambda  \cdot P_\var^{\ell,\lambda}(L),
\end{equation}
 where
\begin{equation}
\tr_{\uBr(\var)} p_\lambda  \; = 
  \begin{tikzpicture}[smallnodes,baseline={([yshift=-0.5ex]0,0)}]
    \draw[int,->] (0,0) arc (180:0:0.3cm) 
  node[yshift=0.17cm,above] {$\smash{\lambda}$}  arc (360:180:0.3cm);
  \end{tikzpicture} = P^\lambda_\var (\unknot) \in \End_{\uBr(\var)}(\emptysequence).  
\label{eq:9}
\end{equation}
In particular, if \(\tr_{\uBr(\var)}p_\lambda\) is non-zero then \(P^{\ell,\lambda}_\var(L)\) is well-defined and can be obtained by division from \(P^\ell_\var(L)\), whence the name ``reduced''.

\begin{lemma}\label{lem:2}
The invariants \(P^{\ell,\lambda}_\beta(L)\) and \(P_d^{\ell,\lambda}(L)\)  are well defined.
\end{lemma}
\begin{proof}
Notice first that if \(d>0\) and \(\lambda\) is a partition with at most \(d\) rows, then \(\tr_{\uBr(d)} p_\lambda \) is the quantum dimension of the irreducible \(U_q(\gl_d)\)--module with highest weight corresponding to \(\lambda\), hence it is nonzero. 

Now, for \(\beta\) generic, \(\tr_{\uBr(\beta)} p_\lambda\) is always non-zero (since it has to specialize for \(\beta=d\) to \(\tr_{\uBr(d)}(p_\lambda)\), which for \(d\gg 0\) is non-zero). Hence \(P^{\ell,\lambda}_\beta(L)\) is equal to \(P^{\ell}_\beta(L)\) divided by \(\tr_{\uBr(\beta)} p_\lambda\), and so it is well-defined.

In the case \(\var=d \in \Z\), by construction \(P^{\ell,\lambda}_d(L)\) is the specialization of \(P^{\ell,\lambda}_\beta(L)\) at \(\beta=d\), hence also \(P^{\ell,\lambda}_d(L)\) is well-defined.
\end{proof}

Notice that, as follows from the proof of the lemma, for \(\beta\) generic the reduced invariants do not give any more information, since we always have
\begin{equation}
  \label{eq:16}
  P^{\ell,\lambda}_\beta(L) = \frac{P^{\ell}_\beta(L)}{\tr_{\uBr(\beta)} p_\lambda}.
\end{equation}
On the other hand, in the specialized case \(\var=d\) it often happens that \(\tr_{\uBr(d)} p_\lambda = 0\), and hence the invariant \(P^\ell_d(L)\) is zero as long as one of the strands is labeled by such a \(\lambda\), while the reduced invariant \(P^{\ell,\lambda}_d(L)\) may be non-zero (cf.\ also Remark~\ref{rem:3} below).

Let us denote by \(H_{m|n}\) the set of hook partitions of type \(m|n\) (i.e.\ the partitions \(\lambda\) with \(\lambda_{m+1} \leq n\)). Then we have the following counterpart of Proposition~\ref{prop:1}:

\begin{prop}\label{prop:2}
    Let \(m,n \in \Z_{\geq 0}\) and let \(d=m-n\). 
Suppose \(\lambda \in H_{m|n}\).
Then \(P^{\ell,\lambda}_d(L)\) is the \(\lambda\)--reduced \(\ell\)--labeled \(\gl_{m|n}\) link invariant (given by the Reshetikhin-Turaev construction). In particular, this link invariant only depends on \(m-n\).
\end{prop}

\begin{proof}
  The proof is analogous to the proof of Proposition~\ref{prop:1}. 
Since \(\lambda \in H_{m|n}\), 
the image of \(p_\lambda\) in \(\catRep_{m|n}\) is non-zero (it projects onto one copy of the simple \(U_q(\gl_{m|n})\)--module \(L(\lambda)\) labeled by \(\lambda\)). In particular, \(\calG_{m|n}\) induces an isomorphism between the idempotent truncation \(p_\lambda \End_{\uBr(d)}(\uparrow^{\otimes \abs {\lambda}}) p_\lambda\) and \(\End_{U_q(\gl_{m|n})}(L(\lambda))\), which are both naturally identified with \(\C(q)\). Hence the claim follows from the commutativity of \eqref{eq:4}.
\end{proof}

\begin{corollary}\label{cor:1}
  The link invariant \(P^{\underline \onepartition,\onepartition}_0(L)\) is (up to rescaling) the \empha{Alexander polynomial} of \(L\).
\end{corollary}

\begin{proof}
  This follows from the proposition above together with \cite[theorem~4.10]{SarAlexander}. Alternatively, one can argue that, up to rescaling, \(P^{\underline \onepartition,\onepartition}_0\) satisfies the skein relations of the Alexander polynomial (see for example \cite[(4.23) and (4.24)]{SarAlexander}).
\end{proof}

\begin{remark}\label{rem:3}
  In the above proposition it is crucial to assume that \(m,n\) are big enough so that \(\lambda \in H_{m|n}\).
This makes a big difference with  Proposition~\ref{prop:1}. Indeed, Proposition~\ref{prop:1} implies that the labeled \(\gl_{m|n}\) link invariant vanishes as long as one strand is labeled by a partition \(\lambda\) such that  \(\lambda \notin H_{m-k|n-k}\) for some \( k \geq \min \{m,n\}\). For example, the \(\gl_{2|1}\) link invariant always vanishes if one strand is labeled by a partition with more than one row. On the other side, Proposition~\ref{prop:2} does not imply that the \(\lambda\)--reduced \(\gl_{m|n}\) link invariant vanishes if the \(\gl_{m-1|n-1}\) does. In particular, it does not imply that the \(\lambda\)--reduced \(\gl_{2|1}\) invariant vanishes if \(\lambda\) has more than one row.
\end{remark}

\providecommand{\bysame}{\leavevmode\hbox to3em{\hrulefill}\thinspace}
\providecommand{\MR}{\relax\ifhmode\unskip\space\fi MR }
\providecommand{\MRhref}[2]{%
  \href{http://www.ams.org/mathscinet-getitem?mr=#1}{#2}
}
\providecommand{\href}[2]{#2}


\begin{thebibliography}{10}

\bibitem{MR1501429}
J.~W. Alexander, \emph{Topological invariants of knots and links},
  \href{http://dx.doi.org/10.2307/1989123}{Trans. Amer. Math. Soc.} \textbf{30}
  (1928), no.~2, 275--306.

\bibitem{MR2998810}
J.~Comes and B.~Wilson, \emph{Deligne's category
  {$\underline{\rm{Rep}}(GL_\delta)$} and representations of general linear
  supergroups},
  \href{http://dx.doi.org/10.1090/S1088-4165-2012-00425-3}{Represent. Theory}
  \textbf{16} (2012), 568--609.

\bibitem{MR3181742}
R.~Dipper, S.~Doty, and F.~Stoll, \emph{The quantized walled {B}rauer algebra
  and mixed tensor space},
  \href{http://dx.doi.org/10.1007/s10468-013-9414-2}{Algebr. Represent. Theory}
  \textbf{17} (2014), no.~2, 675--701.

\bibitem{HOMFLY}
P.~Freyd, D.~Yetter, J.~Hoste, W.~B.~R. Lickorish, K.~Millett, and A.~Ocneanu,
  \emph{A new polynomial invariant of knots and links},
  \href{http://dx.doi.org/10.1090/S0273-0979-1985-15361-3}{Bull. Amer. Math.
  Soc. (N.S.)} \textbf{12} (1985), no.~2, 239--246.

\bibitem{GPT}
N.~Geer, B.~Patureau-Mirand, and V.~Turaev, \emph{Modified quantum dimensions
  and re-normalized link invariants},
  \href{http://dx.doi.org/10.1112/S0010437X08003795}{Compos. Math.}
  \textbf{145} (2009), no.~1, 196--212.

\bibitem{2013arXiv1304.3481G}
E.~{Gorsky}, S.~{Gukov}, and M.~{Stosic}, \emph{{Quadruply-graded colored
  homology of knots}}, ArXiv e-prints (2013),
  \href{http://arxiv.org/abs/1304.3481}{{\ttfamily arXiv:1304.3481}}.

\bibitem{MR873212}
A.~Gyoja, \emph{A {$q$}-analogue of {Y}oung symmetrizer},
  \href{http://projecteuclid.org/euclid.ojm/1200779724}{Osaka J. Math.}
  \textbf{23} (1986), no.~4, 841--852.

\bibitem{Jones}
V.~F.~R. Jones, \emph{A polynomial invariant for knots via von {N}eumann
  algebras}, \href{http://dx.doi.org/10.1090/S0273-0979-1985-15304-2}{Bull.
  Amer. Math. Soc. (N.S.)} \textbf{12} (1985), no.~1, 103--111.

\bibitem{PT}
J.~H. Przytycki and P.~Traczyk, \emph{Conway algebras and skein equivalence of
  links}, \href{http://dx.doi.org/10.2307/2046716}{Proc. Amer. Math. Soc.}
  \textbf{100} (1987), no.~4, 744--748.

\bibitem{QS2}
H.~{Queffelec} and A.~{Sartori}, \emph{{Mixed quantum skew Howe duality and
  link invariants of type A}}, ArXiv e-prints (2015),
  \href{http://arxiv.org/abs/1504.01225}{{\ttfamily arXiv:1504.01225}}.

\bibitem{MR1036112}
N.~Y. Reshetikhin and V.~G. Turaev, \emph{Ribbon graphs and their invariants
  derived from quantum groups},
  \href{http://projecteuclid.org/euclid.cmp/1104180037}{Comm. Math. Phys.}
  \textbf{127} (1990), no.~1, 1--26.

\bibitem{SarAlexander}
A.~Sartori, \emph{The {A}lexander polynomial as quantum invariant of links},
  \href{http://dx.doi.org/10.1007/s11512-014-0196-5}{Ark. Mat.} \textbf{53}
  (2015), no.~1, 177--202.

\bibitem{TVW}
D.~{Tubbenhauer}, P.~{Vaz}, and P.~{Wedrich}, \emph{Super $q$-{Howe} duality
  and web categories}, ArXiv e-prints (2015),
  \href{http://arxiv.org/abs/1504.05069}{{\ttfamily arXiv:1504.05069}}.

\end{thebibliography}
\end{document}